\documentclass[12pt]{article}  
\usepackage{color,amssymb,amsthm,amsmath,amsxtra,amsfonts,breqn,
fullpage,graphicx,hyperref,float,appendix,mathtools,eqparbox, 
lscape,pdfpages,comment,framed,chngpage}

\usepackage[most]{tcolorbox}
\usepackage[utf8]{inputenc}
\usepackage{tikz}
\usetikzlibrary{shapes,arrows}

\usepackage{fancyhdr}
\usepackage{xcolor}
\definecolor{dark-red}{rgb}{0.4,0.15,0.15}
\definecolor{dark-blue}{rgb}{0,0,0.45}
\definecolor{dark-olive-green}{rgb}{0.333,0.42,0.184}
\definecolor{magenta}{rgb}{1.,0.,1.}
\hypersetup{colorlinks, linkcolor={dark-blue},citecolor={blue}, urlcolor={blue}}

\definecolor{gold}{rgb}{0.85,0.66,0.0}

\newtheorem{theorem}{Theorem}

\newtheorem{example}{Example}[section]

\usepackage[ddmmyyyy,hhmmss]{datetime}

\begin{document}

\title{Rational Base Descent: A Deterministic Algorithm for Factoring Structured Semiprimes}
\author{Sam Blake\\ {\small\texttt{samuel [dot] thomas [dot] blake [at] gmail [dot] com} }}
\date{\today}
\maketitle

\begin{abstract}
We present a special-purpose algorithm for factoring semiprimes $N = pq$ in which one prime factor satisfies $p \approx c\,(a/b)^n$ for positive integers $a, b, c, n$ with $a > b$ and $\gcd(a,b) = 1$. Given the correct parameters $(a, b)$, the algorithm isolates a factor in $\mathcal{O}(\log^3 N)$ time when $a/b$ is bounded away from $1$, and the cofactor $q$ is unconstrained beyond a mild size bound. We describe a search strategy over $(a, b)$ using primitivity filters, give a complexity analysis showing that the method poses no threat to balanced RSA semiprimes, and provide a \texttt{gmpy2}-based Python implementation.
\end{abstract}

\section{Introduction}
Integer factorisation---decomposing a composite integer into its prime factors---is among the oldest problems in computational number theory. The Fundamental Theorem of Arithmetic guarantees that every integer greater than one is either prime or has a unique factorisation into primes, and for centuries the construction of factor tables was a manual exercise: Cataldi tabulated factors up to 800 in 1603, Chernac extended this to 1,020,000 in 1811, and D.~N.~Lehmer's tables~\cite{lehmer1909}, published over a century ago, reached 10,017,000.\\

Modern interest in the problem is driven by cryptography. The presumed difficulty of factoring large semiprimes underlies the RSA cryptosystem~\cite{rsa}, whose security rests on the asymmetry between multiplying two primes $p$ and $q$ and recovering them from the product $N = pq$~\cite{cohn}.\\

Factoring algorithms are typically classified as general-purpose or special-purpose. General-purpose algorithms have runtimes depending primarily on the size of $N$, independent of the structure of its factors. The Continued Fraction Factorisation method (CFRAC)~\cite{morrison1975} was the first to achieve sub-exponential complexity, and the General Number Field Sieve (GNFS)~\cite{buhler1993} remains the asymptotically fastest known algorithm for arbitrary $N$.\\

Special-purpose algorithms instead exploit specific structure in the prime factors. Pollard's $p-1$ algorithm succeeds when $p-1$ is smooth; the Elliptic Curve Method (ECM)~\cite{lenstra1987} generalises this idea. Fermat's method is fast when $p$ and $q$ are close to one another~\cite{lehman1974}, and in 1895 Lawrence extended it to the case where $p/q$ is well-approximated by a known rational $u/v$~\cite{lawrence1895}.\\

More recently, Hart~\cite{hart2012} gave a one-line factoring algorithm: a Fermat-style variant whose heuristic complexity is $\mathcal{O}(N^{1/3})$ on general inputs, but which becomes very fast on a sparse class of semiprimes in which \emph{both} prime factors lie near integer powers of a shared base. Specifically, Hart shows that semiprimes of the form $\text{nextprime}(c^a \pm d_1)\times \text{nextprime}(c^b \pm d_2)$, with $c, d_1, d_2$ small and $a, b$ relatively close, can be factored by his method even when the primes have many thousands of digits.\\

The present paper extends this fast-factorable class along two independent axes. First, the integer base $c$ is replaced by a rational base $a/b$, which is dense in $\mathbb{R}_{>0}$ rather than sparse, so the set of approximating bases at any fixed scale is dramatically larger. Second, structure is required only on the prime factor $p$: writing $p \approx c\,(a/b)^n$, the cofactor $q$ is unconstrained beyond a mild size bound. The mechanism is also different: rather than searching for square residues in the style of Fermat, our algorithm performs a deterministic descent in which $Q_k = \lfloor (b/a)\,Q_{k-1}\rfloor$ converges to a known integer multiple of $q$, which is then recovered by a small gcd search.

\section{The Rational Base Descent Algorithm}

The algorithm takes as input a composite integer $N$ and a coprime pair $(a, b)$ with $a > b \ge 1$, representing the rational base $a/b$. It proceeds as follows.

\begin{enumerate}
    \item[\textbf{R1.}] \textbf{[Initialise.]} Set $Q \leftarrow N$.
    \item[\textbf{R2.}] \textbf{[Descend.]} Replace $Q$ by $\lfloor (b\, Q)/a \rfloor$. If $Q < 2$, terminate and return failure: the factor $p$ is not adequately approximated by $a/b$.
    \item[\textbf{R3.}] \textbf{[Search.]} Let $J = \lceil 2 + a/(a-b) \rceil$. For each integer $j \in [-J, J]$:
    \begin{enumerate}
        \item Compute $g \leftarrow \gcd(Q - j, N)$.
        \item If $1 < g < N$, return $g$.
    \end{enumerate}
    \item[\textbf{R4.}] \textbf{[Loop.]} If no factor was found, return to step R2.
\end{enumerate}

\section{Analysis}

The correctness of the algorithm rests on the fact that, after the appropriate number of descent steps, the working variable $Q_n$ lies within a bounded distance of the integer multiple $c\,q$. The theorem below makes this precise.

\begin{theorem}\label{main_theorem}
Let $N = p\,q$ be a semiprime with distinct primes $p$ and $q$, and let $a, b, c, n$ be positive integers with $a > b$ and $\gcd(a,b) = 1$. Define
$$\Delta \;=\; p - \left\lfloor c \left(\frac{a}{b}\right)^n \right\rfloor.$$
If $(a/b)^n > q$ and
$$|\Delta| \;<\; \frac{(a/b)^n}{q},$$
then the Rational Base Descent algorithm, run with the parameters $(a, b)$, returns a non-trivial factor of $N$.
\end{theorem}

\begin{proof}
We show that after exactly $n$ descent iterations of step R2 the integer $\eta := Q_n - c\,q$ satisfies $|\eta| \le J = \lceil 2 + a/(a-b) \rceil$, so the search in step R3 detects it.

Write the descent step as $Q_k = (b/a)\,Q_{k-1} - \epsilon_k$, where $\epsilon_k = \{(b/a)\,Q_{k-1}\} \in [0,1)$ is the fractional part discarded at iteration $k$. Unrolling the recurrence from $Q_0 = N$ gives
\begin{equation}
Q_n = \left(\frac{b}{a}\right)^n N - E_n, \qquad E_n = \sum_{j=1}^{n} \left(\frac{b}{a}\right)^{n-j} \epsilon_j. \label{eqn_Qn}
\end{equation}
Each $\epsilon_j \in [0,1)$ and $b/a < 1$, so the truncation error $E_n$ is bounded by the partial geometric sum
$$0 \le E_n < \sum_{i=0}^{n-1} \left(\frac{b}{a}\right)^i \;<\; \sum_{i=0}^{\infty} \left(\frac{b}{a}\right)^i \;=\; \frac{a}{a-b}.$$

Let $\delta_c = \{c(a/b)^n\} \in [0,1)$, so that $\lfloor c(a/b)^n\rfloor = c(a/b)^n - \delta_c$ and
$$p = c\left(\frac{a}{b}\right)^n - \delta_c + \Delta.$$
Substituting into~\eqref{eqn_Qn} via $N = p\,q$:
\begin{align*}
Q_n &= \left(\frac{b}{a}\right)^n p\,q - E_n \\
&= \left(\frac{b}{a}\right)^n \!\left(c\left(\frac{a}{b}\right)^n - \delta_c + \Delta\right)\!q - E_n \\
&= c\,q + \Delta \left(\frac{b}{a}\right)^n q - \delta_c \left(\frac{b}{a}\right)^n q - E_n.
\end{align*}
Thus
$$\eta = Q_n - c\,q = \Delta \left(\frac{b}{a}\right)^n q - \delta_c \left(\frac{b}{a}\right)^n q - E_n.$$

We bound each term using the hypotheses:
\begin{enumerate}
    \item $|\Delta| < (a/b)^n/q$ gives $\left|\Delta (b/a)^n q\right| < 1$.
    \item $(a/b)^n > q$ gives $(b/a)^n q < 1$, and $\delta_c \in [0,1)$, so $\left|\delta_c (b/a)^n q\right| < 1$.
    \item $|E_n| < a/(a-b)$.
\end{enumerate}
By the triangle inequality,
$$|\eta| \;<\; 2 + \frac{a}{a-b}.$$
Since $\eta = Q_n - c\,q$ is an integer, $|\eta| \le \lceil 2 + a/(a-b)\rceil = J$, so step R3 tests $j = \eta$ and computes $\gcd(Q_n - \eta, N) = \gcd(c\,q,\, p\,q) = q\,\gcd(c, p)$. Since $p \approx c\,(a/b)^n$ with $(a/b)^n > q \ge 2$, we have $p > c$, and as $p$ is prime this forces $\gcd(c, p) = 1$. Hence $\gcd(Q_n - \eta, N) = q$, a non-trivial factor of $N$.
\end{proof}

\section{Searching the Parameter Space}

A naive enumeration of pairs $(a, b)$ scales quadratically in the bound on $a + b$, so practical use of the algorithm requires a search strategy that avoids redundant work.

\subsection{Primitivity Filters}
Some rational bases are algebraically equivalent and can be safely skipped. A base $a/b$ is \emph{primitive} if it cannot be reduced or expressed as a power of a smaller base. We enforce two filters during enumeration:
\begin{itemize}
    \item \textbf{Coprimality.} We require $\gcd(a, b) = 1$. A Euclidean check before invoking the descent rejects all non-coprime pairs. The density of coprime pairs of integers is $6/\pi^2 \approx 0.6079$, so this filter eliminates roughly 39.2\% of the candidate space.\footnote{See, e.g., Hardy and Wright, \textit{An Introduction to the Theory of Numbers}, \S18.5.}
    \item \textbf{Perfect-power redundancy.} If $a = a_0^k$ and $b = b_0^k$ for some integer $k \ge 2$, then $(a/b)^n = (a_0/b_0)^{kn}$, so any vulnerability detectable with the base $a/b$ is also detectable using the primitive base $a_0/b_0$ at a different exponent. We therefore skip such pairs.
\end{itemize}

\section{Asymptotic Complexity}

The cost of one execution of the algorithm is the number of descent iterations multiplied by the per-iteration cost.

During descent, $Q_k$ shrinks by a factor of approximately $b/a$ per step, so reaching the target multiple $c\,q$ from $Q_0 = N$ requires $n \approx \log_{a/b}(p)$ iterations. For balanced semiprimes with $p \approx \sqrt{N}$, this is $n = \mathcal{O}(\log N)$. Each iteration performs at most $2J + 1$ gcd computations, where $J = \lceil 2 + a/(a-b)\rceil$, and the schoolbook Euclidean algorithm computes a gcd in $\mathcal{O}(\log^2 N)$ bit operations. The total cost is therefore
$$\mathcal{O}\!\left(\frac{a}{a-b} \cdot \log^3 N\right).$$
For any fixed rational base $a/b$ bounded away from $1$, the factor $a/(a-b)$ is constant and the algorithm runs in strictly polynomial time $\mathcal{O}(\log^3 N)$. As $a/b \to 1$ the cost grows linearly in $a/(a-b)$, and the favourable regime is bases with a moderate ratio.\\

We now estimate the size of the parameter space that must be searched to factor a generic semiprime. By the theory of Diophantine approximation, a typical real number $X$ admits rational approximations $a/b$ with $|X - a/b| \approx 1/b^2$. Setting $X = (p/c)^{1/n}$ so that $p = c\,X^n$, and writing $X = a/b + \epsilon$ with $\epsilon \approx 1/b^2$, a first-order expansion gives
$$p = c\left(\frac{a}{b}\right)^n + c\,n\left(\frac{a}{b}\right)^{n-1}\epsilon + O(\epsilon^2) \;\approx\; c\left(\frac{a}{b}\right)^n + \frac{c\,n\,a^{n-1}}{b^{n+1}},$$
so the typical magnitude of $\Delta = p - \lfloor c(a/b)^n\rfloor$ is approximately $c\,n\,a^{n-1}/b^{n+1}$. The hypothesis $|\Delta| < (a/b)^n / q$ of Theorem~\ref{main_theorem} then becomes
$$\frac{c\,n\,a^{n-1}}{b^{n+1}} \;<\; \frac{a^n}{b^n\,q}, \qquad \text{i.e.,} \qquad c\,n\,q \;<\; a\,b.$$\\

For an unstructured RSA semiprime with $p \approx q \approx \sqrt{N}$ and small $c, n$, this inequality requires $a\,b \gtrsim \sqrt{N}$. The number of coprime pairs $(a, b)$ with $a\,b \le M$ is $\Theta(M \log M)$, so enumerating such pairs up to $a\,b = \sqrt{N}$ requires $\tilde{\mathcal{O}}(\sqrt{N})$ candidates---exponential in $\log N$ and far worse than the General Number Field Sieve~\cite{buhler1993}. The Rational Base Descent algorithm is therefore no threat to balanced RSA moduli; its value lies entirely in the special case where the prime factor genuinely admits a small rational-power approximation.

\section{Examples}
\begin{example}
When $(a/b)^n \gg q$, the bound on $\Delta$ is large, and the gap between $p_{\text{min}} = \left\lfloor c\, (a/b)^n\right\rfloor - \Delta$ and $p_{\text{max}} = \left\lfloor c\, (a/b)^n\right\rfloor + \Delta$ can be enormous. For instance, with $p_{\text{min}}$, $p_{\text{max}}$ both 866-bit primes and $q$ a 255-bit prime, our algorithm factors both
{\footnotesize
\begin{multline*}
N_{\text{min}} = p_{\text{min}}\,q = 13434917067328449643383271289062122492729438008563207396013386022436439625786\\
814860310635\textcolor{gray}{0926488498087072838850794255593800606037417118877419838277025360771246065263900265}\\
\textcolor{gray}{3123642403389641779880233968185505491600758327359172114787487297101417389933443643227737228011}\\
\textcolor{gray}{2100845364029259802969568473368054409653616559493278947418005745364854816116209}\hfill
\end{multline*}
}
\noindent and 
{\footnotesize
\begin{multline*}
N_{\text{max}} = p_{\text{max}}\,q = 13434917067328449643383271289062122492729438008563207396013386022436439625786\\
814860310635\textcolor{gray}{9462325028510580369721472224479765378713002323091102360779837041712476003336356036}\\
\textcolor{gray}{2466222889720725560742981816345311565868826870566666947258686151144184980492031366549985004825}\\\textcolor{gray}{3134034455058178126780155420375132055020929262098390932828248906338838243332477}\hfill
\end{multline*}
}
with $a/b = 22/7$ in 5113 calls to Euclid's algorithm, despite $p_{\text{max}} - p_{\text{min}} \approx 2.7 \times 10^{178}$ and $N_{\text{max}} - N_{\text{min}} \approx 8.5 \times 10^{254}$. (The digits in which $N_{\text{min}}$ and $N_{\text{max}}$ differ are shown in gray.)
\end{example}

\begin{example}
The algorithm can partially factor certain Cunningham numbers. For $N = 2^{1041} + 1$ with $a/b = 2$, in less than a tenth of a second we obtain the 105-digit factor
\begin{multline*}
2866873269987589389513526119127608675995706236460351404671986049233653595110\cdots\\
60601008752319138765710819329, \hfill
\end{multline*}
which is then easily factored further into
\begin{multline*}
3 \times 9556244233291964631711753730425362253319020788201171348906620164112178650\cdots\\
3686867002917439712921903606443.\hfill
\end{multline*}
The largest of these prime factors does not appear in the prime factor tables of the Cunningham Project~\cite{cunningham,brillhart,wagstaff}.
\end{example}

\begin{example}
The algorithm also handles certain base-10 repunits~\cite{beiler}. For $N = (10^{2224}-1)/9$ with $a/b = 10$, in less than a tenth of a second we obtain the 140-digit factor
\begin{multline*}
188888888888888888888888888888888888888888888888888888888888888888888888888\cdots\\
88888888888888888888888888888888888888888888888888888888888888887, \hfill
\end{multline*}
which factors further as
\begin{multline*}
17 \times 119124859925363 \times 93272815750404359428200304166180738539061812063838459731\cdots\\
40973454310325710523151012873614569275960380714495413018734336834397.\hfill
\end{multline*}
\end{example}

\begin{example}
In 2012, Wagstaff~\cite{wagstaff2012} factored the base-173 repunit
$$\frac{173^{173} - 1}{173 - 1} = 347\times 685081 \times 161297590410850151 \times \text{P}176 \times \text{P}184,$$
where $\text{P}xxx$ denotes a prime with \textit{xxx} decimal digits. Applying our algorithm to
$$\frac{{153}^{153} + 1}{153 + 1}$$
with $a/b = 153$ quickly yields the factor
\begin{multline*}
170507014772455614072877177618780045330205133686049732535688397703000187902\cdots\\
65943188524709552486240402548858137,\hfill
\end{multline*}
giving the full prime factorisation
\begin{multline*}
\frac{{153}^{153} + 1}{153 + 1} = 13\times 409 \times 1789 \times 2647 \times 3061 \times 36451 \times 132949\times \\
1071392089 \times 11689042696587973 \times 25015188869871173 \times \\
3581210233293795917 \times \text{P}55 \times \text{P}194.
\end{multline*}
A larger example, almost certainly out of reach of the elliptic-curve method~\cite{lenstra1987}, is
\begin{multline*}
\frac{{366}^{183}+1}{366+1} = 103 \times 1297 \times 4759 \times 40993 \times 23952871 \times \\
5711029231 \times 265841463337 \times \text{P}146 \times \text{P}280.
\end{multline*}
The primes P55, P194, P146, and P280 are listed in the Appendix.
\end{example}

\section{Open Questions}

\begin{enumerate}
\item Is there a more efficient way to search through the space of $a/b$ representations of $p$?
\item Without knowing $p$, can we construct an integer $m$ such that $m\,p \approx c\,(a/b)^n \pm \Delta$, and then run the algorithm on $m\,N$? For example, $4779846103 \times \text{next\_prime}(2^{304})$ can be factored because $24473 \times \text{next\_prime}(2^{304})$ is close to a power of $3$.
\item Can a gradient-descent approach, applied to the rational-base representation of $N$, be used to recover the $a/b$ representation of $p$? For instance, the base-3 representation of
\begin{multline*}
N = p\,q = 14046133535443172227430326880457961558090460914355488360756493
\end{multline*}
is
\begin{multline*}
\noindent\textcolor{gray}{101201111101112101112100002212121221}00000000000000000000000000\textcolor{gray}{20221200001}\hfill\\
\textcolor{gray}{11220012000020102101210221100210011010022100122220111222}_3,\hfill
\end{multline*}
and the long run of zeros suggests $p \approx 3^n$.
\end{enumerate}

\newpage
\section{Challenge}

As an exercise for the cryptologically inclined reader, the following RSA-encrypted ciphertext $c$ contains a deliberate backdoor and can be broken using the algorithm presented in this paper. The message $c=$
\scriptsize
\begin{verbatim}
1814651558954106218382820026102280437671368367089800692084713095469343...
9790142400605047350523420256388241978645880720770970691940969664014449...
7487493547807943762502518167875215042507410360622552300911078217635608...
0887997102096285507800358419360885775596518636029091371178373992552096...
8622936577315796875988232583687214883888423098977624910410387860547324...
3100154799340490665701745842849161129853985769344005668634371076527914...
7214877747286461688225822696019525443314063817824945242857263753976642...
6089946816180283102898233637188686051762255774304098438442162372944935...
441911381235587297921195954096229860063709180529054299071,
\end{verbatim}
\normalsize
was encrypted using RSA with a 2048-bit public-key modulus $n = $
\scriptsize
\begin{verbatim}
2718022296796546774049853535203404108646031915851663329077780870611416...
3744845322132609618536958174809588920840107120935709320881925676005056...
1442230333011684333489839416774679760022895557020439538683012451397663...
2522502436945795318786473193615886185475646387774096668300777178598694...
9162347938828689500845728036191241787726511296084493467932612591032288...
8869558069630018710582245039859261405571781226467429263986838011424621...
5046937228443293700066476984349745568744964585088683517037751359691649...
3406542248048194462147618414261602864341399915066904968166598805918288...
235109620693254221810328224289006630686023813438006131353,
\end{verbatim}
\normalsize
and public-key exponent $e = 65537$.\footnote{Hint: once the public-key modulus $n$ has been factored as $n = p\,q$, compute Euler's totient $\phi(n) = (p-1)(q-1)$, determine the private exponent $d \equiv e^{-1} \pmod{\phi(n)}$, and recover the (ASCII-encoded) plaintext as $m \equiv c^d \pmod{n}$.}

\bibliographystyle{abbrv}

\appendix
\section{GMP-Based Python Implementation of the Rational Base Descent Algorithm}

The arithmetic underlying the algorithm routinely involves integers of more than 1000 bits, so a fixed-precision implementation is unworkable. The Python~3 implementation below uses GNU Multiple Precision Arithmetic (GMP) via the \texttt{gmpy2} binding, which keeps the inner loop running at C-level speed~\cite{gmp}.

\scriptsize
\begin{tcolorbox}[colframe=white,breakable]
\begin{verbatim}
import math
import gmpy2
from gmpy2 import mpz, is_prime

def generate_primitive_bases(max_sum):
    """Yield primitive coprime rational bases (a, b) ordered by a + b."""
    for s in range(3, max_sum + 1):
        for b in range(1, s // 2 + 1):
            a = s - b
            # Filter 1: coprimality
            if gmpy2.gcd(a, b) != 1:
                continue

            # Filter 2: skip perfect-power redundancies
            is_perfect_power = False
            for k in range(2, int(math.log2(a)) + 2):
                _, exact_a = gmpy2.iroot(a, k)
                _, exact_b = gmpy2.iroot(b, k)
                if exact_a and exact_b:
                    is_perfect_power = True
                    break

            if not is_perfect_power:
                yield (mpz(a), mpz(b))

def rational_base_descent(N, a, b, verbose=False):
    """Run the Rational Base Descent algorithm for the given (a, b).
    Returns a non-trivial factor of N, or 1 on failure."""
    Q = mpz(N)
    # Search radius J = ceil(2 + a/(a-b)), computed in pure integer arithmetic.
    search_limit = 2 + (a + (a - b) - 1) // (a - b)

    while True:
        Q = gmpy2.f_div(gmpy2.mul(b, Q), a)
        if Q < 2:
            return mpz(1)

        # Test j = 0 once, then j = 1, 2, ..., search_limit on each side.
        g = gmpy2.gcd(Q, N)
        if 1 < g < N:
            if verbose:
                print(f"[+] Factor isolated! Base a/b = {a}/{b}")
            return g

        for j in range(1, int(search_limit) + 1):
            g = gmpy2.gcd(Q - j, N)
            if 1 < g < N:
                if verbose:
                    print(f"[+] Factor isolated! Base a/b = {a}/{b}")
                return g
            g = gmpy2.gcd(Q + j, N)
            if 1 < g < N:
                if verbose:
                    print(f"[+] Factor isolated! Base a/b = {a}/{b}")
                return g

def factor_rational_base(N, max_search_sum=1000, verbose=False):
    """Drive the search over (a, b) until a factor of N is found."""
    N = mpz(N)
    if N < 2 or is_prime(N):
        return mpz(1)

    for a, b in generate_primitive_bases(max_search_sum):
        factor = rational_base_descent(N, a, b, verbose)
        if factor != 1:
            return factor

    return mpz(1)
\end{verbatim}
\end{tcolorbox}
\normalsize

\section{Prime factors from the examples}

\scriptsize
\begin{verbatim}
P55 = 6101264910861118599771515106379012943430400007766431341

P146 = 329196142441661574959637081423802218048242445408370309709554587...
4755453931389097023045175105068284254297334161008053329334947236590907...
0730027298473

P194 = 459831305373287624123499753161175453473937443042531100796506875...
4060524490410439502840647915749628498061171811274505844208461951218455...
3847398956309608916765363097365321183418791612373727394676877

P280 = 114345884924653345364772578144643176920559117942637004057824535...
1128934022852893853641130689593814423820074132054371876789339299166332...
0350194620608260957677417484507875744570119308040990990428576013534116...
9704566712544384843478977315999906035947385436280052942220084197314853...
9413623
\end{verbatim}
\normalsize

\end{document}